\documentclass{article}
\usepackage{amsmath,amsthm,amsfonts,amssymb,amscd, mathrsfs,mathtools}
\usepackage{latexsym}
\usepackage{amsmath}
\usepackage{amsthm}
\usepackage{amsfonts,subfig,float}
\usepackage{lineno,hyperref,latexsym}
\usepackage[draft,ulem={normalem,normalbf}]{changes}
\usepackage{comment}

\usepackage{enumerate}
 \usepackage[nospace]{cite}
\usepackage{cleveref}
\definechangesauthor[name={Haiying Shan},color=blue]{SHY}
\definechangesauthor[name={Fei Fei, Wang}, color=orange]{WFF}

\usepackage{tikz}

\usetikzlibrary{calc}

\modulolinenumbers[5]

\newtheorem{thm}{Theorem}[section]
\newtheorem{lem}{Lemma}[section]
\newtheorem{definition}{Definition}[section]

\newtheorem{proof*}{proof}[section]
\newtheorem{cor}{Corollary}[section]
\newtheorem{con}{Conjecture}[section]

\newtheorem{question}{Question}
\newtheorem{conj}{Conjecture}

\usepackage{indentfirst}

\title{On some properties of the $\alpha$-spectral radius of the $k$-uniform hypergraph\thanks{Supported by the Fundamental Research Funds for the Central Universities  (No. 22120180254).}}
\author{  Feifei Wang$^1$, Haiying Shan$^2$\thanks{Corresponding author},  Zhiyi Wang\\[5pt]
{School of Mathematical Sciences, Tongji University, Shanghai, P. R. China} }

\begin{document}

\maketitle

\begin{abstract}

In this paper we show how the $\alpha$-spectral radius changes under the edge grafting operations on connected $k$-uniform hypergraphs.
We characterize the extremal hypertree for $\alpha$-spectral radius among $k$-uniform non-caterpillar hypergraphs with given order, size and diameter.
we also characerize the second largest $\alpha$-spectral radius among all $k$-uniform supertrees on $n$ vertices by two methods.

\begin{flushleft}
{\em Keywords：} Hypergraph; $\alpha$-spectral radius; Supertree;  Non-caterpillar supertree 
\end{flushleft}
\end{abstract}

\section{Introduction}

Spectral graph theory studies connections between combinatorial properties of graphs and the eigenvalues of matrices associated to the graph, such as the adjacency matrix, the Laplacian matrix and the signless Laplacian matrix. In 2017, the concept of $\mathcal A_\alpha$ matrix of a graph $G$ was put forward by Nikiforov \cite{nikiforov2017merging} which is hybrid of $\mathcal A(G)$ and $\mathcal D(G)$ similar to the signless Laplacian matrix, where $\mathcal A(G), \mathcal D(G)$ are the adjacency matrix and degree matrix of graph $G$, respectively. The study of  $\mathcal A_\alpha$ matrix has started attracting attention of researchers recently (see \cite{nikiforov2018alpha,lin2018note,guo2018alpha}).

For positive integers $n$ and $k$, a real tensor $\mathcal { A } = \left( a _ { i _ { 1 } i _ { 2 } \cdots i _ { k} } \right)$  of order $n$ and dimension $k$
refers to a multidimensional array (also called hypermatrix) with entries $a _ { i _ { 1 } i _ { 2 } \cdots i _ { k} }$   such that $a _ { i _ { 1 } i _ { 2 } \cdots i _ { k } } \in \mathbb { R }$ for all $i _ { 1 } , i _ { 2 } , \ldots , i _ { k } \in [ n ] : = \{ 1,2 , \ldots , n \}$.

In 2005, Qi \cite{MR2178089}  introduced the definition of eigenvalues of a tensor. According to the general product of tensors defined in \cite{shao2013general}, the eigenequation for tensor $\mathcal { A }$ can be written in the following form:
\[
	\mathcal { A }\boldsymbol{x} = \lambda \boldsymbol{x}^{[k-1]}.
\]
$\lambda$ is called an \emph{eigenvalue} of $\mathcal { A }$ , $\boldsymbol{x}$ is called an \emph{eigenvector} of $\mathcal { A }$  corresponding to the eigenvalue $\lambda$, where $\boldsymbol { x } ^ { [ r - 1 ] } = \left( x _ { 1 } ^ { r - 1 } , x _ { 2 } ^ { r - 1 } , \ldots , x _ { n } ^ { r - 1 } \right) ^ { \mathrm { T } }$. The spectral radius of $\mathcal { A }$ is the largest modulus of its eigenvalues, denoted by $\rho(\mathcal { A })$.

Analogous to graph theory, adjacency tensors, Laplacian tensors, and signless Laplacian tensors have been introduced in hypergraph theory (see \cite{cooper2012spectra,MR3194370}).

The adjacency, Laplacian and signless Laplacian tensor of a uniform hypergraph has become the key subject in spectral hypergraph theory. For more details, please refer to \cite{hu2014eigenvectors,qi2017tensor} and references therein.  In \cite{lin2018alpha}, the conception of $\alpha$-spectrum of graph is generalized to hypergraphs.

The edge grafting operation for graphs was usually considered in the study on graph variants (e.g. spectral radius, energy and so on).

In \cite{nikiforov2018alpha}, Nikiforov and Rojo  presented {conjectures and question} concerning $\alpha$-spectral radius and edge grafting operation on a simple graph.
\begin{conj}[see \cite{nikiforov2018alpha} Conjecture 18, Conjecture 19]\label{conj1}
Let $\alpha \in [0, 1)$ and \linebreak $s=0,1$. If $q\geq 1$ and $p\geq q+2$, then
$$\rho_{\alpha}(G_{p,s,q}(u,v))<\rho_{\alpha}(G_{p-1,s,q+1}(u,v)).$$
\end{conj}
The conjecture is already confirmed by Guo in \cite{Guo2018On}. Furthermore,
Guo in \cite{guo2018alpha} showed that Conjecture \ref{conj1} also holds for connected $k$-uniform hypergraph when $s=0$.

\begin{question}(see \cite{nikiforov2018alpha} Problem 20 ) \label{question1} For which connected graphs G the following statement is true:

Let $\alpha\in[0,1)$ and let $u$ and $v$ be non-adjacent vertices of $G$ of degree at least $2$. If $q\ge 1$ and $p\ge q+2$, then $\rho_{\alpha}(G_{p,q}(u,v))<\rho_{\alpha}(G_{p-1,q+1}(u,v))$.
\end{question}
The above statement is confirmed true by Lin in \cite{lin2018note} for the type of graphs $G_{p,s,q}(u,v)$ when $p-q\ge max\{s+1,2\}$. In this paper,
 we show the statement also true for connected $k$-uniform hypergraph.

A supertree is a hypergraph which is both connected and acyclic.  A supertree is called a caterpillar if the removal of all pendant edges results in a loose path.
Otherwise, it is called a non-caterpillar.
Let $NC(m)$ denote the set of $k$-uniform non-hyper-caterpillars with $m$ edges. Let $C(m,d)$ and $NC(m,d)$ denote, respectively, the set of $k$-uniform hyper-caterpillars and  non-hyper-caterpillars with $m$ edges and diameter $d$.
It is easy to see that if $NC(m,d)$ is nonempty, then  $m\ge 6$ and $4\le d\le m-2$.

The spectral radii of  $k$-uniform supertrees  have been studied by many authors (see \cite{li2016extremal,MR3462996,lin2018note}).
In \cite{yuan2017ordering,xiao2017maximum}, Yuan et al. determined the top ten supertrees with the maximum spectral radii. And in \cite{guo2018alpha}, Guo et al. determined the unique non-hyper-caterpillar with maximum adjacency spectral radius among ${NC}(m,d)$. They use the relationship $\rho(G^k)=\rho^{\frac{2}{k}}(G)$. But up to now, we have not found the relationship between $\rho_{\alpha}(G)$ and $\rho_{\alpha}(G^k)$. So we should use another new methods to research $\rho_{\alpha}(G^k)$.

Note that the supertrees with the first ten largest spectral radii are caterpillar supertrees. It is natural to ask  that:
\begin{enumerate}[(a).]
	\item Whether those supertrees also are supertrees with the maximum $\alpha$-spectral radius;\label{a}
	\item Which non-caterpillar supertrees have the maximum  adjacency spectral or $\alpha$-spectral radius.\label{b}
\end{enumerate}

The remaining of this paper is organized as follows. In Section 2, notations and some definitions about tensors and hypergraphs are given. In Section 3, applying the results about the perturbation of spectral radii of the hypergraph under moving edges and 2-switching operations, we present a generalization  of  the  result  related to Conjecture \ref{conj1} reported in \cite{guo2018alpha}. In Section 4, we determine the  $k$-uniform non-caterpillar supertrees on $n$ vertices in $NC(m,d)$ with the first two largest  $\mathcal{A}_\alpha$ spectral radii in $NC(m,d)$ and $NC(m)$. In Section 5,  we characterize  the supertree with the second largest $\alpha$-spectral radius among all $k$-uniform supertrees on $n$ vertices.

\section{Preliminaries}

Firstly, we introduce some notations and conceptions used in this paper.

Throughout this paper, we focus on simple $k$-uniform hypergraphs with \linebreak $k \geq  3$. A $k$-uniform hypergraph $H=(V(H),E(H))$ consists of the vertex set $V(H)$ (a finite set) and the edge set $E(H)$ which is a collection of $k$-subsets of $V(H)$.

For a subset $U\subset V(H)$, we denote by $E_H[S]$ the set of edges $\{e\in E(H)|S\cap e\neq \varnothing\}$. For a vertex $i\in V(H)$, we simplify $E_H[i]$ as $E_H(i)$. It is the set of edges containing the vertex $i$, i.e., $E_H(i)=\{e\in E(H)|i\in e\}$. The cardinality $|E_H(i)|$ of the set $E_H(i)$ is defined as the degree of the vertex $i$, which is denoted by $\deg_H(i)$, we simplify $\deg_H(i)$ by $d_H(i)$ or $d_i$.
 A vertex of degree one is called a pendant vertex, otherwise it is called non-pendant vertex.  A vertex with
degree 3 or greater is called a branching vertex. A hyperedge $e$ is called pendant  edge if all vertices of $e$, except at most one, are of degree one. A hyperedge $e$ is called branching edge if there are at lease three vertices with degree 2 or greater in $e$.

Let $G = (V,E)$ be a simple graph. The $k$-th power hypergraph of $G$ is the $k$-uniform hypergraph resulting from adding $k-2$ new vertices to each edge of $G$. (see \cite{hu2013cored})

The $k$-th power of an ordinary tree was called a $k$-uniform hypertree in \cite{hu2013cored}. It is clear that any k-uniform hypertree is a supertree.

If $|e_i\cap e_j|\in \{0,s\}$ for all edges $e_i\neq e_j$ of a hypergraph $G$, then $G$ is called an $s$-hypergraph. A simple graph is a $2$-uniform $1$-hypergraph. Note that $1$-hypergraphs here are also called linear hypergraphs \cite{Bretto2013Hypergraph}. So, all power hypergraphs are linear hypergraph. And a supertree is a linear hypergraphs; otherwise, if two edges $e_1,e_2$ have two vertices $v_1,v_2$ in common, then $v_1e_1v_2e_2v_1$ is a cycle.

According to \cite[p.~146]{qi2017tensor}, a loose path (or linear path) $\mathbb{P}$ of length $d$ is a hypergraph with distinct edges $e_1,e_2,\ldots,e_d$  such that:

\[
	|e_i \cap e_{j}|=\begin{cases}
		1, &\text{ if } |i-j|=1,\\
			0, &\text{otherwise}.
	\end{cases}
\]

If $e_{i-1} \cap e_{i}=\{v_i\}$  for $i=1,\ldots,d-1$ and
$v_0 \in e_1\backslash \{v_1\}, v_d \in e_d\backslash \{v_{d-1}\}$, the aforementioned loose path is denoted by $v_0e_1v_1e_2\ldots e_d v_{d}$ or $\mathbb{P}_d$ for brevity. It is easy to see that a loose path with $d$ edges is the $k$-th power of the ordinary path with $d$ edges.

Let $\mathbb{P}$ is a  loose path from $u$ to $v$ in $H$. If
 $\deg_H(u),\deg_H(v) \geq 3$  and $\deg_H(w)=\deg_\mathbb{P}(w)$ for any $w \in V(\mathbb{P})\backslash \{u,v\}$, $\mathbb{P}$ is called an internal path of $H$. If $\deg_H(u)\geq 3$ and $\deg_H(w)=\deg_\mathbb{P}(w)$ for any $w \in V(\mathbb{P})\backslash \{u\}$, $\mathbb{P}$ is called a pendant path of $H$ on vertex $u$.

For $X \subset V ( G )$
and $y \in V ( G ) \backslash X$, the distance between $X$ and $y$, denoted by $d_G(y,X)=\min\{d_G(y,x)|x\in X\}$, the distance $d_G(X,Y)=\min\{d_G(x,y)|x\in X,y\in Y\}$.

\section{section name}

\begin{definition}[\cite{cooper2012spectra}]\label{def3}
The adjacency tensor $\mathcal A(G)=(a_{i_1\dots i_k})$ of a $k$-uniform hypergraph $G$ is defined to be a $k$-th order $n$ dimensional non-negative tensor with entries $a_{i_1\dots i_k}$ such that $$a_{i_1\dots i_k}=
\begin{cases}
\frac{1}{(k-1)!}& \text{if $\{i_1,i_2,\dots,i_k\} \in E(G)$},\\
0& \text{otherwise}.
\end{cases}$$
\end{definition}

\begin{definition}[\cite{lin2018alpha}]
For $0\leq \alpha < 1$, the $\mathcal{A}_\alpha$ tensor of  a $k$-uniform hypergraph $G$ with $V(G)=\{v_1,v_2,\ldots,v_n\}$ is defined as follows:
$$\mathcal A_{\alpha}(G)=\alpha \mathcal D(G)+(1-\alpha)\mathcal A(G).$$
Where $\mathcal D(G)$ is the degree tensor of $G$, which is a $n$-dimensional diagonal
tensor, with the $i$-th diagonal entry as the degree  of vertex $v_i$.
\end{definition}
The definition of $\mathcal{A}_\alpha$ tensor of hypergraph can be viewed as a generalization of adjacency tensor and signless Laplacian tensor of hypergraph.
Obviously,
$$\mathcal A(G)=\mathcal A_0(G) \text{ and } \mathcal Q(G)=2\mathcal A_{\frac{1}{2}}(G).$$

For connected hypergraph $G$, the  $\mathcal{A}_\alpha$ tensor of $G$ always weakly  irreducible nonnegative and symmetric tensor. According to Perron-Frobenius Theorem of tensor, $\mathcal{A}_\alpha$  has a unique positive eigenvector $\boldsymbol{x}$ corresponding to the spectral radius $\rho(\mathcal{A}_\alpha)$ with $\sum _ { i = 1 } ^ { n } x _ { i } ^ { k } = 1$. We say that $\rho(\mathcal{A}_\alpha)$ is the $\alpha$-spectral radius of $G$, denoted by $\rho_\alpha(G)$. Such a positive eigenvector is called the principal eigenvector of $\mathcal{A}_\alpha$.  For more details of the Perron-Frobenius Theorem on nonnegative tensor, see \cite[Chapter ~3]{qi2017tensor} and references therein.

\begin{thm}[see Theorem 2 in \cite{MR3045233}]\label{thm1} Let $\mathcal{A}$ be a nonnegative symmetric tensor of order $k$ and dimension $n$, denote $\mathcal{R}_{+}^n = \{\boldsymbol{x}\in \mathcal{R}^n|\boldsymbol{x}\ge 0\}$. Then we have
\begin{eqnarray}
\rho(\mathcal{A})=max\{\boldsymbol{x}^\mathrm{T}\mathcal{A} \boldsymbol{x}|\boldsymbol{x}\in \mathcal{R}_{+}^n, \sum_{i=1}^n x_i^k=1\}.  \label{3}
\end{eqnarray}
Furthermore,  $ \boldsymbol{x} \in \mathcal{R}_{+}^n $  with $ \sum_{i=1}^n x_i^k=1 $  is an optimal solution of the above optimization problem if and only if it is an eigenvector of $\mathcal{A}$ corresponding to the eigenvalue $\rho(\mathcal{A})$.
\end{thm}

For a vector $\boldsymbol{x}$ of dimension $n$ and subset $U\subseteq \left[n\right]$, we write $$x_U=\prod_{i\in U} x_i.$$
From Definition \ref{def3}, using the general product of tensors defined by Shao in \cite{shao2013general}, we have

 \begin{eqnarray}
 &\boldsymbol{x}^T \mathcal{A}(G)\boldsymbol{x}=\sum_{e\in E(G)} k x_e.\nonumber \\
 &\boldsymbol{x}^T\mathcal{A}_\alpha(G)\boldsymbol{x}=\sum_{e\in E(G)}(\alpha \sum_{u\in e}x_u^k+(1-\alpha)kx_e). \label{eq:shao_spe}
 \end{eqnarray}

Let $\boldsymbol{x}$ be the principal eigenvector of $\mathcal A_{\alpha}(G)$.  By Theorem \ref{thm1}, the eigenequation of $\mathcal{A}_\alpha(G)$ for $\rho(\mathcal{A}_\alpha(G))$ can be written as follows:

\begin{equation}\label{eq:eigncomp}
	\rho_{\alpha}(G)x_v^{k-1}=\alpha d_vx_v^{k-1}+(1-\alpha)\sum_{e\in E_G(v)}x_{e\backslash\{v\}}.
\end{equation}

\begin{thm}
(see \cite{guo2018alpha}) For $k\ge 2$, let $G$ be a $k$-uniform hypergraph with $u,v_1,\dots,v_r\in V(G)$ and $e_1,\dots,e_r\in E(G)$ for $r\ge 1$ such that $u \notin e_i$ and $v_i\in e_i$ for $i=1,\dots,r$, where $v_1,\dots,v_r$ are not necessarily distinct. Let $e'_i=(e_i\backslash \{v_i\})\cup \{u\}$ for $i=1,\dots,r$. Suppose that $e'_i\notin E(G)$ for $i=1,\dots,r$. Let $G'=G-\{e_1,\dots,e_r\}+\{e'_1,\dots,e'_r\}$. Let $\boldsymbol{x}$ be the $\alpha$-Perron vector of $G$. If $x_u\ge max\{x_{v_1},\dots,x_{v_r}\}$, then $\rho_{\alpha}(G')> \rho_{\alpha}(G)$. \label{thm 4}
\end{thm}

\begin{cor}\label{cor1}
Let $u_1,u_2$ are non-pendant vertices in an edge of connected uniform hypergraph $H$ with $|E_H(u_i)\backslash E_H(\{u_1,u_2\})|\ge 1$ for $i=1,2$. Let $H'$ be the hypergraph obtained from $H$ by moving edges $E_H(u_2)\backslash E_H(\{u_1,u_2\})$ from $u_2$ to $u_1$ and $H\ncong H'$, then $$\rho_{\alpha}(H)<\rho_{\alpha}(H').$$
\end{cor}
\begin{proof}Let $H''$ be the hypergraph obtained from $H$ by moving edges \linebreak$E_H(u_1)\backslash E_H(\{u_1,u_2\})$ from $u_1$ to $u_2$. Moreover, $H'\cong H''$. It is obvious that either $x_{u_1}\ge x_{u_2}$ or $x_{u_2}\ge x_{u_1}$ holds, thus either $H'$ or $H''$ satisfies the conditions of Theorem \ref{thm 4}. So we can apply Theorem \ref{thm 4}, we have  $\rho_{\alpha}(H)<max\{\rho_{\alpha}(H'),\rho_{\alpha}(H'')\}=\rho_{\alpha}(H')$.
\end{proof}

\section{The generalization of edge grafting theorem for hypergraph on $\mathcal{A}_\alpha$ spectral radius}

Consider a connected $k$-uniform hypergraph $G$ with  two vertices (not necessarily distinct) $u,v \in V(G)$. Let {$G_{u,v}(p,q)$} be the hypergraph obtained by attaching the pendant hyperpaths $\mathbb{P}_p$ to $u$ and $\mathbb{P}_q$ to $v$ (See Fig. \ref{fig:grafting}).

\begin{figure}
\centering
\includegraphics[page=1,width=.45\textwidth]{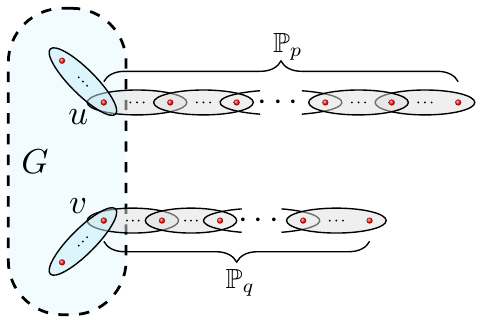}
 \caption{The hypergraph $G_{u,v}(p,q)$}
 \label{fig:grafting}
\end{figure}

For non-negative integers $a,b,c,d$ with $a+b=c+d$, we say
that the graph $G_{u,v}(c,d)$ is obtained from $G_{u,v}(a,b)$ by an
edge grafting operation on the two relevant pendant paths of
$G_{u,v}(a,b)$.
The edge grafting operation for graphs was usually considered in the study on graph variants.

The 2-switch operation is a useful tool for study in graph theory, especially in terms of degree sequence. In \cite{xiao2017maximum,guo2018alpha}, the perturbation of the adjacency spectral radius and $\alpha$-spectral radius  of a hypergraph under 2-switch operation  is studied.

\begin{definition}[2-switching operation, see \cite{xiao2017maximum}]
Let $e_1,e_2$ be two edges of $k$-uniform hypergraph $G=(V,E)$. If $U_1,U_2$ are $r$-subsets of $e_1,e_2$ respectively with $1 \le r <k $
 and $k$-sets $e_1^{\prime}=(e_1\cup U_2)\backslash U_1 ,e_2^{\prime}=(e_2\cup U_1)\backslash U_2 \notin E(G)$. Let $E^{\prime}=(E\cup \{e_1^{\prime},e_2^{\prime}\})\backslash \{e_1,e_2\}$. Then we say that hypergraph $G^{\prime}=(V,E^{\prime})$ is obtained from $G$ by 2-switching operation $e_1 \xrightleftharpoons[U_2]{\;U_1} e_2$.
\end{definition}

It is clear that the degrees of the vertices remain unchanged when a 2-switch is applied to a hypergraph. The following theorem illustrates the effect of 2-switch operation on the $\alpha$-spectral radius of hypergraph.

\begin{thm}\label{thm 1}
Let $G$ be a connected $k$-uniform hypergraph and $e, f \in E(G)$.  Let $U_1\subset e, U_2=e\backslash U_1, V_1\subset f, V_2=f\backslash V_1$ with $1\le |U_1|=|V_1|\le k-1$. Suppose that $e'=U_1\cup V_2$ and $f'=V_1\cup U_2$ are $k$-subsets of $V(G)$ and not in  $E(G)$.  Let $G'=G-\{e,f\}+\{e',f'\}$.
Let $\boldsymbol{x}$ be the principal eigenvector of $\mathcal A_{\alpha}(G)$.
If $x_{U_1}\ge x_{V_1}$ and $x_{U_2}\le x_{V_2}$, then $\rho_{\alpha}(G)\le \rho_{\alpha}(G')$.
Moreover, the equality holds iff  $x_{U_2}=x_{V_2}$ and $x_{U_1}=x_{V_1}$.
\end{thm}
\begin{proof}
By Theorem \ref{thm1}  and Formula \eqref{eq:shao_spe}, we have
\begin{align*}
&\rho_{\alpha}(G')-\rho_{\alpha}(G)\\
\ge&\boldsymbol{x}^T\mathcal A_{\alpha}(G')\boldsymbol{x}-\boldsymbol{x}^T\mathcal A_{\alpha}(G)\boldsymbol{x}\\
=&k(1-\alpha)(x_{U_1}x_{V_2}+x_{U_2}x_{V_1}-x_{U_1}x_{U_2}-x_{V_1}x_{V_2})\\
=&k(1-\alpha)(x_{U_1}-x_{V_1})(x_{V_2}-x_{U_2})\\
\ge &0.
\end{align*}
So, $\rho_{\alpha}(G)\le \rho_{\alpha}(G')$.

From Formula \eqref{eq:eigncomp} we have:
\begin{equation}\label{eq:4}
	\mathcal A_{\alpha}(G')x_v^{k-1}-\mathcal A_{\alpha}(G)x_v^{k-1}= \begin{cases}
0 & \text{ if $v \notin e\cup f$};  \\
x_{V_1\backslash\{v\}}(x_{U_2}-x_{V_2}) & \text{ if $v \in V_1$;}   \\
x_{V_2\backslash\{v\}}(x_{U_1}-x_{V_1}) &  \text{ if $v \in V_2$;}  \\
x_{U_1\backslash\{v\}}(x_{V_2}-x_{U_2}) & \text{ if $v \in U_1$;}   \\
x_{U_2\backslash\{v\}}(x_{V_1}-x_{U_1}) &  \text{ if $v \in U_1$.}  \\
\end{cases}
\end{equation}

Since $\rho_{\alpha}(G)\le \rho_{\alpha}(G')$, we have $\mathcal A_{\alpha}(G')\boldsymbol{x}\geq \mathcal A_{\alpha}(G)\boldsymbol{x}$
and the  equality holds iff $\rho_{\alpha}(G)= \rho_{\alpha}(G')$.

By Perron-Frobenious Theorem， $\boldsymbol{x}$ is positive vector. Then from Formula \eqref{eq:4}, we have $\rho_{\alpha}(G)=\rho_{\alpha}(G')$ if and only if $x_{U_1}=x_{V_1}, x_{U_2}=x_{V_2}$.

\end{proof}

The following corollaries can be derived from Theorem \ref{thm 1}.
\begin{cor}[\cite{guo2018alpha}]\label{thm 2}
Let $G$ be a connected $k$-uniform hypergraph with $k\ge 2$, and $e$ and $f$ be two edges of $G$ with $e\cap f=\varnothing$. Let $\boldsymbol{x}$ be the principal eigenvector of $\mathcal A_{\alpha}(G)$. Let $U\subset e$ and $V\subset f$ with $1\le |U|=|V|\le k-1$. Let $e'=U\cup (f\backslash V)$ and $f'=V\cup (e\backslash U)$. Suppose that $e',f'\notin E(G)$. Let $G'=G-\{e,f\}+\{e',f'\}$. If $x_U\ge x_V,\ x_{e\backslash U}\le x_{f\backslash V}$ and one is strict, then $\rho_{\alpha}(G)<\rho_{\alpha}(G')$.
\end{cor}

\begin{cor}\label{cor3.2}Let $G$ be a connected $k$-uniform hypergraph and $\boldsymbol{x}$ be the principal eigenvector of $\mathcal A_{\alpha}(G)$. Suppose that $e,f \in E(G)$ such that $$\{u_1,u_2\} \subset e, \{v_1,v_2\} \subset f \text{  and  }x_{u_1} > x_{v_1}, x_{u_2} \le x_{v_2}.$$
If
$u_i$ is not adjacent to $v_j$ in $G$ for any $i,j \in \{1,2\}$, then there exist $k$-subsets $e',f'$ of $V(G)$ with  $\{u_1,v_2\} \subset e', \{u_2,v_1\} \subset f'$ such that:
$$\rho_{\alpha}(G) < \rho_{\alpha}(G'),$$
where $G'=G-\{e,f\}+\{e',f'\}$.
\end{cor}
\begin{proof}Let \[
	e'= \begin{cases}
	e\cup \{v_2\}\backslash \{u_2\} & \text{if $x_{e\backslash \{u_1,u_2\}}\ge x_{f\backslash \{v_1,v_2\}}$ ,}\\
	f\cup \{u_1\}\backslash \{v_1\} & \text{otherwise};
	\end{cases} $$  and $$f'= \begin{cases}
	f\cup  \{u_2\}\backslash  \{v_2\} & \text{if $x_{e\backslash \{u_1,u_2\}}\ge x_{f\backslash \{v_1,v_2\}}$ ,}\\
	e\cup \{v_1\}\backslash \{u_1\} & \text{otherwise}.
	\end{cases}
\]

If $x_{e\backslash \{u_1,u_2\}}\ge x_{f\backslash \{v_1,v_2\}}$, then take\\
$$U_1=e\backslash \{u_2\}, U_2=\{u_2\},\qquad V_1=f\backslash \{v_2\}, V_2=\{v_2\}.$$
If $x_{e\backslash \{u_1,u_2\}}< x_{f\backslash \{v_1,v_2\}}$, then take
$$U_1=\{u_1\}, U_2=e\backslash \{u_1\},\qquad  V_1=\{v_1\}, V_2=f\backslash \{v_1\}.$$
In either case, we always have
$$e=U_1 \cup U_2, f=V_1\cup V_2, e'=U_1\cup V_2, f'=V_1\cup U_2 $$ and $$x_{U_1}> x_{V_1}, x_{U_2}\le x_{V_2}.$$
Since $u_i$ is not adjacent to $v_j$ in $G$ for any $i,j \in \{1,2\}$, we have $e',f' \notin E(G)$.
Take $G'=G-\{e,f\}+\{e',f'\}$.
From Theorem \ref{thm 2}, we have $\rho_{\alpha}(G') > \rho_{\alpha}(G).$
\end{proof}

The Corollary \ref{cor3.2} is a key ingredient of the proof of main results.

\begin{lem} \label{lem 1}
Let $u e_1 u_1\dots u_p e_{p+1} u_{p+1}$ and $v f_1 v_1 \dots v_{q-2} f_{q-1} v_{q-1}$ be two pendant paths in $G_{u,v}(p+1,q-1)$ at $u$ and $v$, respectively. Let $\boldsymbol{x}$ be the principal eigenvector of $\mathcal A_{\alpha}(G_{u,v}(p+1,q-1))$.
 If $\rho_{\alpha}(G_{u,v}(p,q))\le \rho_{\alpha}(G_{u,v}(p+1,q-1))$, then $x_{u_{p-i}}>x_{v_{q-i-1}}$ for $i=0,\dots,q-1$,
where $u_0=u,v_0=v$.
\end{lem}
\begin{proof}
Assume, by contradiction, that there exists some $i \in \{0,1,\ldots,q-1\}$ such that $x_{u_{p-i}}\leq x_{v_{q-i-1}}$.  Let $j$ be the smallest number of them.

When $j=0$, let $H$ be the $k$-uniform hypergraph obtained from $G_{u,v}(p+1,q-1)$ by moving $e_{p+1}$ from $u_p$ to $v_{q-1}$. Since $H\cong G_{u,v}(p,q)$,  we have $\rho_{\alpha}(G_{u,v}(p,q))=\rho_{\alpha}(H)>\rho_{\alpha}(G_{u,v}(p+1,q-1))$ by Theorem \ref{thm 4}, a contradiction.

When $j \geq 1$, we have $x_{u_{p-j}} \leq x_{v_{q-j-1}}$ and $x_{u_{p-j+1}}>x_{v_{q-j}}$.

According to Corollary \ref{cor3.2}, there exist $e',f'$ such that\\ $\{v_{q-j-1},u_{p-j+1} \} \subset e'$,  $\{u_{p-j},v_{q-j}\}\subset f'$ and $\rho_{\alpha}(G_{u,v}(p+1,q-1)) < \rho_{\alpha}(G')$,\\ where $G'=G-\{e_{p-j+1},f_{q-j}\}+\{e',f'\}\cong G_{u,v}(p,q)$.
This contradicts the condition $\rho_{\alpha}(G_{u,v}(p,q))\le \rho_{\alpha}(G_{u,v}(p+1,q-1))$.
\end{proof}

\begin{thm}\label{thm5}
Let $u,v$ be two non-pendant vertices of hypergraph $G$. If there exist an internal path $\mathbb{P}$ with $s$ length in hypergraph $G_{u,v}(p,q)$ for any $p\ge q\ge 1$,
then we have
 $$\rho_{\alpha}(G_{u,v}(p+1,q-1))<\rho_{\alpha}(G_{u,v}(p,q))\text{\qquad  for } p-q+1\ge s\geq 0.$$
\end{thm}
\begin{proof}
On the contradiction, we assume there exists some positive integers $p,q$ with $p-q+1\ge s\geq 0$ such that
\begin{equation}\label{eq:assumption}
	\rho_{\alpha}(G_{u,v}(p,q))\le \rho_{\alpha}(G_{u,v}(p+1,q-1)).
\end{equation}

Let $\mathbb{Q}=w_0m_1w_1\dots w_{s-1}m_sw_s$  be the internal path of $G_{u,v}(p+1,q-1)$ with $w_0=v$ and $w_s=u$. Let $H$ be the $k$-uniform hypergraph obtained from $G_{u,v}(p+1,q-1)$ by moving $E_G(v)\backslash \{m_1\}$ from $v$ to $u_{p-q+1}$ when $s\ge 1$ or moving $E_G(v)$ from $v$ to $u_{p-q+1}$ when $s=0$.
Let $\boldsymbol{x}$, $\boldsymbol{x}'$  be the principal eigenvectors of $\mathcal A_{\alpha}(G_{u,v}(p+1,q-1))$, $\mathcal A_{\alpha}(H)$, respectively.

Under the assumption \eqref{eq:assumption}, we have $x_{u_{p-i}}>x_{v_{q-i-1}}$ for $i=0,\dots,q-1$ from Lemma \ref{lem 1}. Thus $x_{u_{p-q+1}}>x_v$. By Theorem \ref{thm 4}, we have
\begin{equation}\label{eq:8}
	\rho_{\alpha}(H)>\rho_{\alpha}(G_{u,v}(p+1,q-1)).
\end{equation}
Then $p-q+1>s>1$ holds. Otherwise, we have  $s=p-q+1$ or $s=0$, then both cases $H\cong G_{u,v}(p,q)$.  So $$\rho_{\alpha}(G_{u,v}(p,q))>\rho_{\alpha}(G_{u,v}(p+1,q-1)),$$ which contradicts to the assumption \eqref{eq:assumption}.

Furthermore, we have the following claim.

\noindent{\bfseries{Claim}}:  $x'_{w_i}<x'_{u_{p-q+1-i}}$ for $0\le i\le s$.

\noindent{Proof}. Assume for contradiction that there exist some $i \geq 0$ such that $x'_{w_i}\geq x'_{u_{p-q+1-i}}$.
 Let $j$ be the smallest number of them. Then we have
$j>0$. Otherwise,  since $w_0=v$, $x'_{u_{p-q+1}}\leq x'_v$,  $\rho_{\alpha}(H)<\rho_{\alpha}(G_{u,v}(p+1,q-1))$ holds from Theorem \ref{thm 4}. This contradicts the formula \eqref{eq:8}.

Herefore, $j>0$  and $x'_{w_j}\geq x'_{u_{p-q+1-j}},x'_{w_{j-1}}<x'_{u_{p-q+2-j}}$.

According to Corollary \ref{cor3.2}, there exist $e',m'$ such that\\[1mm] \hspace*{1cm} $\{w_j,u_{p-q+2-j} \} \subset e'$,  $\{w_{j-1},u_{p-q+1-j}\}\subset m'$ and $\rho_{\alpha}(H) < \rho_{\alpha}(G')$, \\[1mm]
where $G'=H-\{m_j,e_{p-q+1-j}\}+\{e',m'\} \cong G_{u,v}(p,q)$. By Formula \eqref{eq:8}, we have $\rho_{\alpha}(G_{u,v}(p+1,q-1))<\rho_{\alpha}(H) < G_{u,v}(p,q)$.

This contradicts the assumption \eqref{eq:assumption}.  \hfill $\blacksquare$\vspace{3mm}

 Let $\widetilde{G}$ be the $k$-uniform hypergraph obtained from $H$ by moving  $E_G(u)\backslash m_s$ from $u$ to $u_{p-q+1-s}$. Then $\widetilde{G}\cong G_{u,v}(p,q)$.
From Claim, $x'_u=x'_{w_s}<x'_{u_{p-q+1-s}}$. So $$\rho_{\alpha}(G_{u,v}(p,q))=\rho_{\alpha}(\widetilde{G})>\rho_{\alpha}(H)>\rho_{\alpha}(G_{u,v}(p+1,q-1))$$  follows from  Theorem \ref{thm 4} and Formula \eqref{eq:8}, also a contradiction.

Therefore, $\rho_{\alpha}(G_{u,v}(p,q))>\rho_{\alpha}(G_{u,v}(p+1,q-1))$.
\end{proof}

Let $G$ be a connected $k$-uniform hypergraph and $u\in V(G)$, let $G_u(p,q)$ be the graph obtained by attaching the paths $\mathbb{P}_p$ and $\mathbb{P}_q$ to $u$.

From Theorem \ref{thm5}, we have the following corollary which  reported in \cite{guo2018alpha}

\begin{cor}[\cite{guo2018alpha}]
For $k\ge 2$, let $G$ be a connected $k$-uniform hypergraph with $|E(G)|\ge 1$ and $u\in V(G)$. For $p\ge q \ge 1$ and $0\le \alpha <1$, we have $\rho_{\alpha}(G_u(p,q))>\rho_{\alpha}(G_u(p+1,q-1))$.    \label{thm6}
\end{cor}

\begin{thm}\label{thm10}
Let $G$ be a connected uniform hypergraph  and $u, v$ be two pendant vertices in a pendant edge $e$ of $G$. If $p\ge q\ge 1$, then $$\rho_{\alpha}(G_{u,v}(p,q))>\rho_{\alpha}(G_{u,v}(p+1,q-1)). $$
\end{thm}
\begin{proof}
On the contradiction, we assume that
\begin{equation}\label{eq:assumption2}
	\rho_{\alpha}(G_{u,v}(p,q))\le \rho_{\alpha}(G_{u,v}(p+1,q-1)).
\end{equation}

As mentioned in Lemma \ref{lem 1}, let $u e_1 u_1\dots u_p e_{p+1} u_{p+1}$ and $v f_1 v_1 \dots v_{q-2} f_{q-1} v_{q-1}$ be the pendant paths in $G_{u,v}(p+1,q-1)$ at $u$ and $v$, respectively. Then by Lemma \ref{lem 1}, we have $x_{u_{p-i}}>x_{v_{q-i-1}}$ for $i=0,\dots,q-1$. Thus $x_{u_{p-q+1}}>x_v$.

Suppose $$e=\{u,v,w_1,w_2,\dots,w_{k-2}\}$$   and $$e_{p-q+1}=\{u_{p-q},u_{p-q+1},w'_1,w'_2,\dots,w'_{k-2}\}. $$

We have $x_{w_1}> x_{w'_1}$. Otherwise, take $G'$ be hypergraph obtained from $G_{u,v}(p+1,q-1)$ by moving all the edges incident with $w_1$ except $e$ from $w_1$ to $w'_1$.
By Theorem \ref{thm 4} and the fact that $G'\cong G_{u,v}(p,q)$, we have\\ $\rho_{\alpha}(G_{u,v}(p,q))> \rho_{\alpha}(G_{u,v}(p+1,q-1))$, a contradiction.

According to Corollary \ref{cor3.2}, there exists $e',f'$ such that\\ $$\{w_1,u_{p-q+1}\}\subset e', \{w'_1, v\}\subset f'$$ and $$\rho_{\alpha}(G_{u,v}(p,q))=\rho_{\alpha}(H')>\rho_{\alpha}(G_{u,v}(p+1,q-1)),$$ where $H'=G-\{e_{p-q+1},e\}+\{e',f'\}$. This contradicts the assumption \eqref{eq:assumption2}.
\end{proof}

\section{Non-caterpillar supertrees  with large $\alpha$-spectral radius}

Let $K_{1,m}$ be the ordinary star with $m$ edges. Let $K_{1,m}^k$ be the $k$th power of $K_{1,m}$. Let the double star $S(a,b)$ be the ordinary tree with $a+b+2$ vertices obtained from an edge $e$ by attaching $a$ pendant edges to one end vertex of $e$, and attaching $b$ pendant edges to the other end vertex of $e$. The $k$th power of $S(a,b)$ is denoted by $S^k(a,b)$.

\begin{figure}[H]%
\centering
\includegraphics[page=5,scale=1.2]{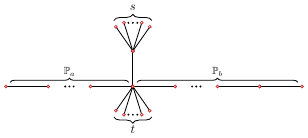}
\caption{The tree $T(s,t;a,b)$}%
       \label{fig:Tstab}%
\end{figure}

Suppose $a,b$ are positive integers with $a+b=d$.  As shown in Fig \ref{fig:Tstab},   $T(s,t;a,b)$  is the tree obtained from $S(s,t)$ by attaching  two  paths of  length $a$ and $b$ to the central vertex of degree  $t+1$  of $S(s,t)$. And $T^k(s,t;a,b)$ denote the $k$th power of  $T(s,t;a,b)$.

Let $$H_1(m,d)=T^k(1,m-d-2;\lfloor{\frac{d}{2}}\rfloor,d-\lfloor{\frac{d}{2}}\rfloor)$$ and $$H_2(m,d)=T^k(m-d-1,0;\lfloor{\frac{d}{2}}\rfloor,d-\lfloor{\frac{d}{2}}\rfloor).$$

\begin{figure}[H]%
\centering
\subfloat[][$H_1(m,d)$]{\includegraphics[page=3,width=.47\textwidth]{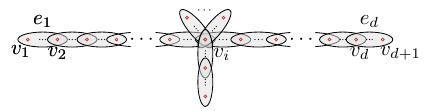}}
\subfloat[][$H_2(m,d)$]{\vspace*{-5mm}
\raisebox{4mm}[0pt][0pt]{\includegraphics[page=4,width=.47\textwidth]{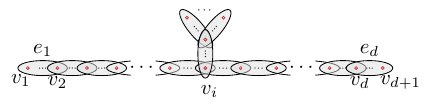}}}
\caption{The hypertrees $H_1(m,d)$ and $H_2(m,d)$  (where $i=\lceil  \frac{d}{2} \rceil+1$)}%
       \label{fig:extremaltrees}%
\end{figure}

In  \cite{guo2018spectral},
Guo and Zhou investigated the adjacency spectral radius of uniform hypertrees and showed that $H_1(m,d)$ be the unique non-hyper-caterpillar with maximum spectral radius among $NC(m,d)$. Next we determine the supertrees with the first two largest $\alpha$-spectral radius among $NC(m,d)$.

\begin{lem}
For any $G\in NC(m,d)$, there exists some $G'\in C(m,d)$ with $\rho_{\alpha}(G')>\rho_{\alpha}(G)$.
\end{lem}
\begin{proof}
Let $\mathbb{P}=v_1 e_1 v_2 \dots v_d e_d v_{d+1}$ be a diametral path of $G$.

Let $u \in V(G)$  be the farthest vertex  from $\mathbb{P}$. Clearly, this distance is at least $2$ (otherwise, $G\in C(m,d)$). Let
$r\in V(\mathbb{P})$ and  $\mathbb{Q}=uf_1vf_2w\dots r$ be the shortest path from $u$ to $\mathbb{P}$. Take $G_1$ to be the $k$-uniform supertree obtained from $G$ by moving all edges containing $v$ except $f_2$ from $v$ to $w$. Obviously, $G_1\ncong G$. Then by Corollary \ref{cor1}, we have $\rho_{\alpha}(G)<\rho_{\alpha}(G_1)$. Repeating this procedure, we arrive at a hyper-caterpillar $G'$ with $\rho_{\alpha}(G)<\rho_{\alpha}(G')$. This completes the proof.
\end{proof}

\begin{definition}
Let $G$ be a $k$-uniform linear hypergraph with $k\ge 3$. Let $e=\{v_1,\dots,v_k\}$ be an edge of $G$ with $d_G(v_i)\ge 2$ for $i=1,\dots,r$, and $d_G(v_i)=1$ for $i=r+1,\dots,k$, where $3\le r\le k$. Let $G'$ be the hypergraph obtained from $G$ by moving all edges containing $v_3,\dots,v_r$ except $e$ from $v_3,\dots,v_r$ to $v_1$. We say $G'$ is obtained from $G$ by Operation I.
\end{definition}
\begin{thm}
(see \cite{guo2018alpha}) If $G'$ is obtained from $G$ by Operation I, then we have $\rho_{\alpha}(G')>\rho_{\alpha}(G)$.   \label{thm3}
\end{thm}

Let $\mathcal{G}$ be a family  of linear hypergraphs  which is invariable under Operation I, namely, for any $G \in \mathcal{G}$, the hypergraph obtained from $G$ by Operation I also in $\mathcal{G}$.
\begin{lem}\label{lem2}
Let $H$ with maximum $\alpha$-spectral radius among $\mathcal{G}$, then $H$ is a power hypergraph.
\end{lem}
\begin{proof}
If $H$ is not a power hypergraph, then there is an edge $f$ with at least three vertices of degree at least $2$, and by Operation I we can get a hypergraph $H'$. Obviously, $H'\in \mathcal{G}$ and by Theorem \ref{thm3}, we have $\rho_{\alpha}(H')>\rho_{\alpha}(H)$, a contradiction. Therefore, $H$ is a power hypergraph.
\end{proof}

\begin{thm}
Let $G\in NC(m,d)$, then we have $$\rho_{\alpha}(G)\le max\{\rho_{\alpha}(H_1(m,d)),\rho_{\alpha}(H_2(m,d))\}.$$ Equality holds if and only if $G\cong H_1(m,d)$ or $G\cong H_2(m,d)$.    \label{thm8}
\end{thm}
\begin{proof}
Let $G$ with maximum $\alpha$-spectral radius among $NC(m,d)$. Let $\boldsymbol{x}$ be the principal eigenvector of $\mathcal A_{\alpha}(G)$. Let $\mathbb{P}=v_1e_1v_2 \dots v_de_dv_{d+1}$ be the diametral path of $G$. From $NC(m,d)\subseteq \mathcal{G}$ and Lemma \ref{lem2}, we have $G$ is a $k$-uniform power hypertree.

\noindent{\bfseries{Claim 1}}. $\forall v\in V(G), d(v,\mathbb{P})\le 2$.

\noindent\itshape{proof}. Assume to the contrary, $\exists u \in V(G),\ d(u,\mathbb{P})=d(u,r)\ge 3,\ r\in V(\mathbb{P})$. Let $\mathbb{Q}=uf_1vf_2w\dots r$ be the shortest path from $u$ to $\mathbb{P}$. Let $G'$ be the $k$-uniform power hypertree obtained from $G$ by moving all edges containing $v$ except $f_2$ from $v$ to $w$. Obviously, $G'\in NC(m,d)$ and $G'\ncong G$. Then by Corollary \ref{cor1}, we have $\rho_{\alpha}(G')>\rho_{\alpha}(G)$, a contradiction. Thus $\forall v\in V(G), d(v,\mathbb{P})\le 2$.\hfill$\blacksquare$\vspace{3mm}

\noindent{\bfseries{Claim 2}}. There only exists a vertex $v_i$ with $d(v_i)\ge 3$, $2\le i\le d$.

\noindent\itshape{proof}. Suppose that there are two vertices  $v_j$ and $v_s$ with $d(v_j),d(v_s)\ge 3$   $(2\le j,s \le d)$ on $\mathbb{P}$. Without loss of generality, assume that $x_{v_j}\ge x_{v_s}$. Let $G'$ be the hypertree obtained from $G$ by moving all edges containing $v_s$ except $e_{s-1}$ and $e_s$ from $v_s$ to $v_j$. Obviously, $G'\in NC(m,d)$. By Theorem \ref{thm 4}, $\rho_{\alpha}(G')>\rho_{\alpha}(G)$, a contradiction. Thus there only exists a vertex $v_i$ with $d(v_i)\ge 3, 2\le i\le d$.\hfill$\blacksquare$\vspace{3mm}

\noindent{\bfseries{Claim 3}}. There exists exactly one vertex $v\in V(G\backslash \mathbb{P})$ with $d(v)\ge 2$.

\noindent\itshape{proof}. From the definition of non-caterpillar supertrees, we have that there must exist vertex $v\in V(G\backslash \mathbb{P})$ with $d(v)\ge 2$. Assume to the contrary, there exists vertices $u,v\in V(G\backslash \mathbb{P})$ with $d(u),d(v)\ge 2$. Without loss of generality, assume that $x_u\ge x_v$. Let $G'$ be the hypertree obtained from $G$ by moving all pendant edges containing $v$ from $v$ to $u$. Obviously, $G'\in NC(m,d)$. By Theorem \ref{thm 4}, $\rho_{\alpha}(G')>\rho_{\alpha}(G)$, a contradiction. Thus there exists exactly one vertex $v\in V(G\backslash \mathbb{P})$ with $d(v)\ge 2$.\hfill$\blacksquare$\vspace{3mm}

Suppose that there are $l$ pendant edges in $G$ and $t$ pendant edges attached on the diametral path $\mathbb{P}$.
From above claims and Theorem \ref{thm 4}, $G$ must be some graph $T^k_{s,t;a,b}$ with $s=l-t-2>0, a \geq 2, b \geq 2$ and $a+b=d$.
Otherwise, we can relocate some pendant edge attached on the diametral path such that the resulting graph $G' \in NC(m,d)$ and $\rho_{\alpha}(G) < \rho_{\alpha}(G')$, a contradiction. For the same reason, $G$ must be  one of $T^k_{s+t,0;a,b}$ and $T^k_{1,s+t-1;a,b}$.

Furthermore, by Corollary \ref{thm6}, we have $G\cong H_1(m,d)$ or $G\cong H_2(m,d)$. Otherwise, by grafting edges on the diametral path, we can
get some $G' \in NC(m,d)$ with $\rho_{\alpha}(G) < \rho_{\alpha}(G')$, a contradiction.  This completes the proof.
\end{proof}

According  to \cite{guo2018alpha}, Guo et al. determined the unique non-hyper-caterpillar with maximum adjacency spectral radius among ${NC}(m)$. So Guo et al. confirmed \eqref{b} for the adjacency spectral and in the following theorem, we determine the supertrees with the first two largest  $\alpha$-spectral radius among ${NC}(m)$. And we confirm \eqref{b} for the $\alpha$-spectral radius.
\begin{thm}
If $G\in NC(m)$ and $0\le \alpha <1$, then $$\rho_{\alpha}(G)\le max\{\rho_{\alpha}(H_1(m,4)),\rho_{\alpha}(H_2(m,4))\}.$$ Equality holds if and only if $G\cong H_1(m,4)$ or $G\cong H_2(m,4)$.
\end{thm}
\begin{proof}
Let $G$ with maximum $\alpha$-spectral radius among $NC(m)$. Let $\boldsymbol{x}$ be the principal eigenvector of $\mathcal A_{\alpha}(G)$. Then by Theorem \ref{thm8}, $G\cong H_1(m,d)$ or $G\cong H_2(m,d)$ for some $d$ with $4\le d\le m-2$. Suppose that $d\ge 5$. Let $\mathbb{P}=v_1e_1v_2\dots v_de_dv_{d+1} $ be a diametral path of $G$. Choose $v_t,v_{t+1}\in \{v_2,v_3,\dots, v_d\}$, without loss of generality, assume that $x_{v_t}\ge x_{v_{t+1}}$. Let $G'$ be the hypertree obtained from $G$ by moving all edges containing $v_{t+1}$ except $e_t$ from $v_{t+1}$ to $v_t$. Obviously, $G'\in NC(m)$. By Theorem \ref{thm 4}, we have $\rho_{\alpha}(G)<\rho_{\alpha}(G')$, a contradiction. It follows that $d=4$, i.e., $G\cong H_1(m,4)$ or $G\cong H_2(m,4)$.
\end{proof}

 We finish this section with the following conjectures:

\begin{con}
For $G\in NC(m,d)$ and $0\le \alpha <1$, then $$\rho_{\alpha}(G)\le \rho_{\alpha}(H_1(m,d)).$$ Equality holds if and only if $G\cong H_1(m,d)$.
\end{con}

\begin{con}
For $G\in NC(m)$ and $0\le \alpha <1$, then $$\rho_{\alpha}(G)\le \rho_{\alpha}(H_1(m,4)).$$ Equality holds if and only if $G\cong H_1(m,4)$.
\end{con}

\section{The second largest $\alpha$-spectral radius of $k$-uniform supertrees}
In this section,  we will characterize the supertree with the second largest $\alpha$-spectral radius among all $k$-uniform supertrees on $n$ vertices by two methods. In \cite{li2016extremal}, Li and Shao showed that for the adjacency spectral radius, the signless Laplacian spectral radius and the incidence
$Q$-spectral radius, $S^k(1,m-2)$ attains uniquely the second largest spectral radius among all $k$-uniform supertrees on $n$ vertices and $m$ edges.  In \cite{guo2018alpha}, Guo and Zhou  determined the hypergraph with the largest $\alpha$-spectral radius among all $k$-uniform supertrees on $n$ vertices.
\begin{thm}
(see \cite{guo2018alpha}) Suppose that $k\ge 2$. If $G$ is a $k$-uniform supertree with $m\ge 1$ edges, then $\rho_{\alpha}(G)\le \rho_{\alpha}(K_{1,m}^k)$ for $0\le \alpha <1$ with equality holds if and only if $G\cong K_{1,m}^k$.
\end{thm}

 Let $\mathbb{T}$ be the set of $k$-uniform supertrees of order $n$. Let $N_2(G)$ be the number of non-pendant vertices of a hypergraph $G$. Denote by $\mathbb{T}_i$ the set of supertrees in $\mathbb{T}$ with $N_2(\mathbb{T})=i$.

\begin{lem}\label{lem4}
Let $T\in \mathbb {T}_{i+1}(n,k)$ be a $k$-uniform supertree of order $n$ with $N_2(T)=i+1\ge 2$. Then there exists a supertree $T'\in \mathbb {T}_i(n,k)$ such that $\rho_{\alpha}(T')>\rho_{\alpha}(T)$.
\end{lem}
\begin{proof}
Let $u,\ v$ be two non-pendant vertices of $T$. Let $\mathbb{P}=ue_1\dots e_sv$ be a shortest path from $u$ to $v$. Let $T_1$ be the hypergraph obtained from $T$ by moving all the edges incident with $u$ (except $e_1$) from $u$ to $v$, and $T_2$ be obtained from $T$ by moving all the edges incident with $v$ (except $e_s$) from $v$ to $u$. Then both $T_1$ and $T_2$ are still supertrees (since they are still connected and have the same number of edges as $T$), and they all have one more pendant vertex than $T$. So we have
$N_2(T_j)=N_2(T)-1=i       (j=1,2)$, note that $T_j\in \mathcal T_i(n,k)\ (j=1,2),$ and $x_{u}\ge x_{v}$ or $x_{v}\ge x_{u}$. Then by Theorem \ref{thm 4} , we have $\rho_{\alpha}(T)<max\{\rho_{\alpha}(T_1),\rho_{\alpha}(T_2)\}$. Take $T'$ to be one of $T_1$ and $T_2$ with the larger spectral radius, we obtain the desired result.
\end{proof}

\begin{lem}       \label{lem3}
Let $a,b,c,d$ be nonnegative integers with $a+b=c+d$. Suppose that $a \le b, c \le d$ and $a < c$, then we have:
$$\rho_{\alpha}(S^k(a,b))> \rho_{\alpha}(S^k(c,d)).  $$
\end{lem}

\begin{proof}
Let $u,\ v$ be the two non-pendant vertices of $S^k(c,d)$ with the degrees $d(u)=c+1$ and $d(v)=d+1$. Let $G'$ be obtained from $S^k(c,d)$ by moving $c-a$ pendant edges from $u$ to $v$, and $G''$ be obtained from $S^k(c,d)$ by moving $d-a$ pendant edges from $v$ to $u$. Then both $G'$ and $G''$ are isomorphic to $S^k(a,b)$.

Here $x_u\ge x_v$ or $x_v\ge x_u$, by Theorem \ref{thm 4}, we have
$$\rho_{\alpha}(S^k(a,b))=max(\rho_{\alpha}(G'),\rho_{\alpha}(G'')) > \rho_{\alpha}(S^k(c,d)).$$
\end{proof}

\begin{thm}
Let $T$ be a $k$-uniform supertree on $n$ vertices with $m$ edges, suppose that $T\ncong K_{1,m}^k$, then we have $$\rho_{\alpha}(T)\le \rho_{\alpha}(S^k(1,m-2))$$  where the equality holds if and only if $T\cong S^k(1,m-2)$.
\end{thm}

\begin{proof} \textbf{(Methods one)}\\
 We use induction on the number of non-pendant vertices $N_2(T)$. Since $T\ncong K_{1,m}^k$, we have $N_2(T) \ge 2$. Now we assume that $T\ncong S^k(1,m-2)$.

If $N_2(T) = 2$, then the two non-pendant vertices $u$ and $v$ of $T$ must be adjacent. Otherwise, all the internal vertices of the path between $u$ and $v$ would be non-pendant vertices other than $u$ and $v$, contradicting $N_2(T) = 2$. And so it can be easily verified that $T\cong S^k(c,d)$ for some positive integers $2 \le c \le d$ with $c+d=m-1$ ($2 \le c $ since $T\ncong S^k(1,m-2))$. So by Lemma \ref{lem3}, we get the desired results.

If $N_2(T) \ge 3$, then by Lemma \ref{lem4}, there exists a supertree $T'\in \mathbb{T}_{i-1}(n,k)$ such that $\rho_{\alpha}(T')>\rho_{\alpha}(T)$. Repeating this procedure, we arrive at a $T''\in T_2(n,k)$ such that $\rho_{\alpha}(T)<\rho_{\alpha}(T'')\le \rho_{\alpha}(S^k(1,m-2))$. This completes the proof.
\end{proof}

\begin{proof}\textbf{(Methods two)}\\
Let $T$ be a $k$-uniform supertree with maximum $\alpha$-spectral radius among $k$-uniform supertrees on $n$ vertices with $m$ edges except the hyperstar $K_{1,m}^k$.

Let $d$ be the diameter of $T$. Since $T\ncong K_{1,m}^k$, we have $d\ge 3$. Suppose that $d\ge 4$. Let $v_0e_1v_1\dots e_dv_d$ be a diametral path of $T$. Choose $v_t,v_{t+1}\in \{v_1,\dots,v_{d-1}\}$, without loss of generality, assume that $x_{v_t}\ge x_{v_{t+1}}$. Let $T_1$ be the $k$-uniform supertree obtained from $T$ by moving the edges incident $v_{t+1}$ except $e_t$ from $v_{t+1}$ to $v_t$. Obviously, $T_1\ncong K_{1,m}^k$. By Theorem \ref{thm 4}, we have $\rho_{\alpha}(T)<\rho_{\alpha}(T_1)$, a contradiction. Thus $d=3$ and $T\cong S^k(c,d)$. By Lemma \ref{lem3}, we have $T\cong S^k(1,m-2)$.
\end{proof}
\begin{cor}
$S^k(1,m-2)$ is the unique supertree which achieves the second maximum spectral radius among all $k$-uniform supertree on $n$ vertices with $m$ edges.
\end{cor}

By this point, we give a positive answer to Question \eqref{a}: the supertrees with the first two largest spectral radii  also are supertrees with the first two largest $\alpha$- spectral radii.

\addcontentsline{toc}{section}{References}

\bibliography{References}

\end{document}